\documentclass[reqno,12pt]{amsart}
\theoremstyle{plain}
\newtheorem{thm}{Theorem}[section]
\newtheorem{cor}[thm]{Corollary} 
\newtheorem{lemma}[thm]{Lemma} 
\newtheorem{prop}[thm]{Proposition}

\newtheorem{proposition}[thm]{Proposition}

\theoremstyle{remark}

\newtheorem{remark}[thm]{Remark}

\theoremstyle{definition}

\usepackage{amscd,amssymb,comment,epic,eepic,euscript,graphics}
\usepackage{enumerate}
\usepackage[initials]{amsrefs}
\usepackage[all]{xy}
\usepackage{color}

\newcount\theTime
\newcount\theHour
\newcount\theMinute
\newcount\theMinuteTens
\newcount\theScratch
\theTime=\number\time
\theHour=\theTime
\divide\theHour by 60
\theScratch=\theHour
\multiply\theScratch by 60
\theMinute=\theTime
\advance\theMinute by -\theScratch
\theMinuteTens=\theMinute
\divide\theMinuteTens by 10
\theScratch=\theMinuteTens
\multiply\theScratch by 10
\advance\theMinute by -\theScratch

\def\today{{\number\day\space
 \ifcase\month\or
  January\or February\or March\or April\or May\or June\or
  July\or August\or September\or October\or November\or December\fi
 \space\number\year}}

\newcommand\Ac{{\mathcal{A}}}
\newcommand\Afr{{\mathfrak A}}
\newcommand\alphah{{\hat\alpha}}
\newcommand\alphat{{\tilde\alpha}}
\newcommand\Bc{{\mathcal{B}}}
\newcommand\Bct{{\widetilde\Bc}}
\newcommand\Cc{{\mathcal{C}}}
\newcommand\conv{\operatorname{conv}}
\newcommand\Cpx{{\mathbf C}}
\newcommand\Dc{{\mathcal{D}}}

\newcommand\Et{{\widetilde E}}
\newcommand\Fbb{{\mathbb F}}
\newcommand\Fc{{\mathcal{F}}}
\newcommand\Ft{{\widetilde F}}
\newcommand\HEu{{\EuScript H}}                   

\newcommand\Ints{{\mathbf Z}}
\newcommand\Mcal{{\mathcal{M}}}
\newcommand\Mcalt{{\widetilde\Mcal}}
\newcommand\Nats{{\mathbf N}}
\newcommand\Nc{{\mathcal{N}}}
\newcommand\NTS{\operatorname{NTS}}

\newcommand\oup{^{\mathrm o}}
\newcommand\psih{{\hat\psi}}
\newcommand\QSS{{\operatorname{QSS}}}
\newcommand\restrict{{\upharpoonright}}
\newcommand\Tc{{\mathcal{T}}}
\newcommand\TQSS{{\operatorname{TQSS}}}
\newcommand\Uc{{\mathcal{U}}}
\newcommand\xh{{\hat x}}
\newcommand\ZQSS{{\operatorname{ZQSS}}}
\newcommand\ZTQSS{{\operatorname{ZTQSS}}}

\begin{document}

\title[Tracial quantum symmetric states]{The simplex of tracial quantum symmetric states}

\author[Dabrowski]{Yoann Dabrowski$^{*}$}
\address{Y.\ Dabrowski,
Universit\'{e} de Lyon,
Universit\'{e} Lyon 1,
Institut Camille Jordan UMR 5208,
43 blvd. du 11 novembre 1918,
F-69622 Villeurbanne cedex,
France}
\email{dabrowski@math.univ-lyon1.fr}
\thanks{\footnotesize $^{*}$Supported in part by ANR grant NEUMANN.
$^\dag$Supported in part by NSF grant DMS-1202660.
$^\ddag$Supported in part by CPDA grant of IITM}

\author[Dykema]{Kenneth J.\ Dykema$^\dag$}
\address{K.\ Dykema, Department of Mathematics, Texas A\&M University,
College Station, TX 77843-3368, USA}
\email{kdykema@math.tamu.edu}

\author[Mukherjee]{Kunal Mukherjee$^\ddag$}
\address{K.\ Mukherjee, Department of Mathematics, Indian Institute of Technology Mad\-ras, 
Chennai  -- 600 036, India}
\email{kunal@iitm.ac.in}


\date{27 October, 2014}

\begin{abstract}
We show that the space of tracial quantum symmetric states of an arbitrary unital C$^*$-algebra
is a Choquet simplex and is a face of the tracial state space of the universal unital C$^*$-algebra free product
of $A$ with itself infinitely many times.
We also show that the extreme points of this simplex are dense, making it the Poulsen simplex
when $A$ is separable and nontrivial.
In the course of the proof
we characterize the centers of certain tracial amalgamated free product C$^*$-algebras.
\end{abstract}

\maketitle

\section{Introduction and description of results}

Quantum exchangeable random variables (namely, random variables whose distributions are invariant for the
natural co-actions of S.\ Wang's quantum permtuation groups~\cite{W98})
were characterized by K\"ostler and Speicher~\cite{KS09} to be those sequences of identically distributed random variables
that are free with respect to the conditional expectation onto their tail algebra (that is, free with amalgamation
over the tail algebra).

In~\cite{DKW}, Dykema, K\"ostler and Williams considered, for any unital C$^*$-algebra $A$,
the analogous notion of quantum symmetric states
on the universal unital free product C$^*$-algebra $\Afr=*_1^\infty A$.
The symbols $\QSS(A)$ denote the compact convex set of all quantum symmetric states on $\Afr$.
The paper~\cite{DKW} contains a convenient characterization of the extreme points of $\QSS(A)$.
Also the compact convex set $\TQSS(A)\subseteq\QSS(A)$
of all tracial quantum symmetric states on $\Afr$ was considered, and
the extreme points of $\TQSS(A)$ were described.
Question~8.8 of~\cite{DKW} asks whether $\TQSS(A)$ is a Choquet simplex (when $A$ has a tracial state, for otherwise
$\TQSS(A)$ is empty).

The main result of this note is that $\TQSS(A)$ is a Choquet simplex whose extreme points are dense.
Thus, when $A$ is separable and nontrivial,
$\TQSS(A)$ is the Poulsen simplex~\cite{LOS78}, which is the unique metrizable simplex whose extreme points
are dense.
In showing this, we also see that $\TQSS(A)$ is
a face of the simplex $TS(\Afr)$
of all tracial states on $\Afr$
and we obtain a better description of the extreme points of $\TQSS(A)$.

Along the way, we prove some technical results that we neeed and that may be useful in other contexts.
In Section~\ref{sec:amalgtr}, we provide a proof (not readily found in the literature) of a well known fact
that natural conditions are sufficient for an amalgamted free product to have a trace.
In Section~\ref{sec:center}, we characterize the centers of certain tracial von Neumann algebra free products
with amalgamation
and we use this to characterize the set of conditional-expectation-preserving traces of von Neumann algebras.
Section~\ref{sec:tail} is short and consists of a technical result about conditional expectations.
Finally, in Section~\ref{sec:Choq}, we prove the main result.

\smallskip
{\bf Acknowledgements.}  K.\ Mukherjee gratefully acknowledges the hospitality and support of the Mathematics
Department at Texas A\&M Univeristy during the Workshop in Analysis and Probability in summer 2013
(funded by a grant from the NSF);
Y.\ Dabrowski and K.\ Dykema are grateful to the Fields Institute for its support during
the Focus Program on Noncommutative Distributions in Free Probability Theory
and to the Mathematisches Forschungsinstitut Oberwolfach for its support during a workshop on $C^*$-algebras;
much of this research was conducted at these three meetings.

\section{Amalgamated free products and tracial amalgamated free products}
\label{sec:amalgtr}

Let $\Dc$ be a von Neumann algebra,
let $I$ be a nonempty set and for every $i\in I$ let $\Bc_i$ be a von Neumann algebra
containing $\Dc$ by a unital inclusion of von Neumann algebras, and suppose $E_i:\Bc_i\to\Dc$ is a normal conditional
expectation with faithful GNS representation.
Let
\[
(\Mcal,F)=(*_\Dc)_{i\in I}(\Bc_i,E_i)
\]
be the von Neumann algebra amalgamated free product.
In the case that the $E_i$ are all faithful, details of this construction were given by Ueda~\cite{U99}, and he showed
that then $F$ is faithful (see p.\ 364 of~\cite{U99}).
Alternatively, and also when the conditional expectations $E_i$ fail to be faithful but do have faithful GNS constructions,
the free product construction may be performed by (a) taking the C$^*$-algebra free product $(\Mcal_0,F_0)$
of the $(\Bc_i,E_i)$
acting on the free product Hilbert C*-module $V$, (b) taking any normal, faithful $*$-representation $\pi$ of $\Dc$ on 
a Hilbert space $\HEu_\pi$, (c) letting $\Mcal$ be
the strong-operator-topology closure of the image of the resulting representation of $\Mcal_0$
on the Hilbert space $V\otimes_\pi\HEu_\pi$ and (d) letting $F:\Mcal\to\Dc$ be compression
by the projection from $V\otimes_\pi\HEu_\pi$ onto the Hilbert subspace $\Dc\otimes_\pi\HEu_\pi$.
The fact that $\Mcal$ is independent of the representation $\pi$ follows from that fact that any two normal faithful
representations of $\Dc$ are related by dilation and compression by a projection in the commutant.

The following result is well known, but since we rely on it, this seems like a good place to give a brief proof.
\begin{proposition}\label{prop:amalgtr}
Suppose $\tau$ is a normal trace on $\Dc$ such that for all $i\in I$, $\tau\circ E_i$ is a trace on $\Bc_i$.
Then $\tau\circ F$ is a trace on $\Mcal$ and is faithful if and only if $\tau$ is faithful.
Furthermore, every normal tracial state on $\Mcal$ that is preserved by $F$ arises in this fashion.
\end{proposition}
\begin{proof}
Since every tracial state $\tau$ on $\Mcal$ that is preserved by $F$ must equal $\tau\restrict_\Dc\circ F$, the last assertion
of the proposition is clearly true.
Moreover, suppose we know that $\tau\circ F$ is a trace;
if we assume also that $\tau$ is faithful, then the GNS representation of $\tau\circ F$
will be faithful;
since it is a trace, if follows that $\tau\circ F$ is itself faithful.
Thus, we need only show that $\tau\circ F$ is a trace.

Let $\Bc_i\oup=\Bc_i\cap\ker E_i$.
Let $m,n\in\Nats$ and let $b_j\in\Bc_{i(j)}\oup$ for $1\le j\le m$ and $c_j\in\Bc_{k(j)}\oup$ for all $1\le j\le n$,
with $i(j)\ne i(j+1)$ and $k(j)\ne k(j+1)$.
If $d\in\Dc$, then by freeness, we have
\begin{equation}\label{eq:Fd}
F(d(c_1c_2\cdots c_n))=0=F((c_1\cdots c_n)d),
\end{equation}
so the composition with $\tau$ is also zero.
We will show by induction on $\min(m,n)$ that
\begin{equation}\label{eq:tauFtr}
\tau\circ F\big((b_m\cdots b_2b_1)(c_1c_2\cdots c_n)\big)
=\tau\circ F\big((c_1c_2\cdots c_n)(b_m\cdots b_2b_1)\big)
\end{equation}
and, furthermore, that the above quantity is zero unless $m=n$ and $i(j)=k(j)$ for all $j$,
in which case it equals
\begin{multline}\label{eq:taubcform}
\tau\circ E_{i(m)}(b_m\,E_{i(m-1)}(b_{m-1}\cdots E_{i(2)}(b_2\,E_{i(1)}(b_1c_1)\,c_2)\cdots c_{m-1})\,c_m) \\
=\tau\circ E_{i(1)}(c_1\,E_{i(2)}(c_2\cdots E_{i(m-1)}(c_{m-1} E_{i(m)}(c_mb_m)\,b_{m-1})\cdots b_2)\,b_1).
\end{multline}
This will suffice to prove the lemma, because the span of $\Dc$ and such elements $b_1\cdots b_m$ is dense
in $\Mcal$.

By freeness, we have
\begin{equation}\label{eq:Fdelta}
F\big((b_m\cdots b_2b_1)(c_1c_2\cdots c_n)\big)
=\delta_{i(1),k(1)}\,F\big((b_m\cdots b_2)E_{i(1)}(b_1c_1)(c_2\cdots c_n)\big).
\end{equation}
If $m=n=1$, then~\eqref{eq:tauFtr} and~\eqref{eq:taubcform} follow from traciality of $\tau\circ E_{i(1)}:\Bc_{i(1)}\to\Cpx$.
If $\min(m,n)=1$ and $\max(m,n)>1$, then the right-hand-side of~\eqref{eq:Fdelta} is zero by~\eqref{eq:Fd},
and by symmetry also $F((c_1c_2\cdots c_n)(b_m\cdots b_2b_1))=0$,
as required.

We may, thus suppose $\min(m,n)>1$ and $i(1)=k(1)$.
Then, using the induction hypothesis (and noting that $\Dc c_2\subseteq\Bc_{k(2)}\oup$),
we have
\begin{align*}
\tau&\circ F\big((b_m\cdots b_2b_1)(c_1c_2\cdots c_n)\big) \\
&=\delta_{i(1),k(1)}\,\tau\circ F\big((b_m\cdots b_2)E_{i(1)}(b_1c_1)(c_2\cdots c_n)\big) \\
&=\delta_{i(1),k(1)}\delta_{m,n}\delta_{i(2),k(2)}\cdots\delta_{i(m),k(m)}  \\
&\quad \cdot \tau\circ E_{i(m)}(b_m\,E_{i(m-1)}(b_{m-1}\cdots E_{i(2)}(b_2E_{i(1)}(b_1c_1)\,c_2)\cdots c_{m-1})\,c_m).
\end{align*}
If $m\ne n$ or if $m=n$ but $i(j)\ne k(j)$ for some $j$, then not only is the above quantity zero but, by symmetry, also
$\tau\circ F((c_1c_2\cdots c_n)(b_m\cdots b_2b_1))$ vanishes.

We may, thus, suppose $m=n>1$ and $i(j)=k(j)$ for all $j$.
Treating $E_{i(1)}(b_1c_1)c_2$ as an element of $\Bc_{k(2)}\oup$,
by the induction hypothesis of~\eqref{eq:taubcform}, we get
\begin{align*}
\tau&\circ E_{i(m)}(b_m\,E_{i(m-1)}(b_{m-1}\cdots E_{i(2)}(b_2E_{i(1)}(b_1c_1)\,c_2)\cdots c_{m-1})\,c_m) \\
&=\tau\circ E_{i(2)}(E_{i(1)}(b_1c_1)c_2\,E_{i(3)}(c_3\cdots E_{i(m)}(c_mb_m)\cdots b_3)b_2) \\
&=\tau\big(E_{i(1)}(b_1c_1)\,E_{i(2)}(c_2\,E_{i(3)}(c_3\cdots E_{i(m)}(c_mb_m)\cdots b_3)b_2)\big) \\
&=\tau\circ E_{i(1)}\big(b_1c_1\,E_{i(2)}(c_2\,E_{i(3)}(c_3\cdots E_{i(m)}(c_mb_m)\cdots b_3)b_2)\big) \\
&=\tau\circ E_{i(1)}\big(c_1\,E_{i(2)}(c_2\,E_{i(3)}(c_3\cdots E_{i(m)}(c_mb_m)\cdots b_3)b_2)b_1\big),
\end{align*}
where in the last equality we have used the traciality of $\tau\circ E_{i(1)}$.
Thus, we have proved the identity~\eqref{eq:taubcform} and that this quantity equals
\[
\tau\circ F((b_m\cdots b_2b_1)(c_1c_2\cdots c_n)).
\]
By symmetry, it is equal also to 
$\tau\circ F((c_1c_2\cdots c_n)(b_m\cdots b_2b_1))$.
\end{proof}

Of course, the result analogous to Proposition~\ref{prop:amalgtr}
for amalgamated free products of C$^*$-algebras, is true by the same proof.

\section{Centers of certain amalgamated free products}
\label{sec:center}

Let $\Dc\subseteq\Bc$ be a unital inclusion of
von Neumann algebras with a normal
conditional expectation $E:\Bc\to\Dc$ whose GNS representation is faithful.
Suppose there is a normal, faithful, tracial state $\tau_\Dc$ on $\Dc$ such that $\tau_\Bc:=\tau_\Dc\circ E$ is a trace on $\Bc$.
The GNS representation of $\tau_\Bc$ is an action of $\Bc$
on the Hilbert space $L^2(\Bc,\tau_\Bc)=L^2(\Bc,E)\otimes_\Dc L^2(\Dc,\tau)$
by multiplication on the left and, thus, the GNS representation of $\tau_\Bc$ is faithful.
Since $\tau_\Bc$ is a trace, it follows that $\tau_\Bc$ itself is faithful and, hence, $E$ must be faithful.

For an element $x$ of a von Neumann algebra, we will let $[x]$ denote the range projection of $x$.
Thus, $[x]$ is the orthogonal projection onto the closure of the range of $x$, considered as a Hilbert space operator,
and it belongs to the von Neumann algebra generated by $x$.
The notation $Z(\Ac)$ means the center of $\Ac$.

\begin{lemma}\label{lem:suppProj}
With $E:\Bc\to\Dc$ and trace $\tau_\Bc$ as above, let
\[
q=q(E)=\bigvee\;\{[E(b^*b)]\mid b\in\ker E\}.
\]
Then $q\in\Dc\cap Z(\Bc)$, and $(1-q)\Bc=(1-q)\Dc$.
\end{lemma}
\begin{proof}
If $b\in\ker E$ and $u$ is a unitary in $\Dc$ then $bu\in\ker E$, and
\[
[E((bu)^*(bu)]=[u^*E(b^*b)u]=u^*[E(b^*b)]u
\]
and we get $u^*qu=q$.
Thus, $q\in Z(\Dc)$.

If $q\not\in Z(\Bc)$, then there would be a partial isometry $v\in\Bc$ so that $0\ne v^*v\le 1-q$ and $vv^*\le q$.
Since $q\in Z(\Dc)$ we get $E(v)=qE(v)(1-q)=0$.
But, since $E$ is faithful, $E(v^*v)\ne0$ and $[E(v^*v)]\le 1-q$, contrary to the definition of $q$.
Thus, we must have $q\in Z(\Bc)$.

If $(1-q)\Bc\ne(1-q)\Dc$, then there would be $b\in(1-q)\Bc\cap\ker E$ with $b\ne0$.
But again, this yields $0\ne E(b^*b)=(1-q)E(b^*b)$, contrary to the choice of $q$.
Thus, we must have
$(1-q)\Bc=(1-q)\Dc$.
\end{proof}

Let
\begin{equation}\label{eq:MF1}
(\Mcal,F)=(*_\Dc)_1^\infty(\Bc,E)
\end{equation}
be the von Neumann algebra free product with amalgamation over $\Dc$ of infinitely many copies of $(\Bc,E)$.
Our main goal in this section is to show that the center of $\Mcal$ is contained in $\Dc$.

Let $\tau=\tau_\Dc\circ F$.
By Proposition~\ref{prop:amalgtr}, $\tau$ is a faithful trace on $\Mcal$.

Let $(\Bc_i,E_i)$ denote the $i$-th copy of $(\Bc,E)$ in the construction of $\Mcal$.
We now describe some standard notation for $\Mcal$ and related objects.
The von Neumann algebra $\Mcal$ is constructed on the Hilbert space $L^2(\Mcal,\tau)$,
and we write $\Mcal\ni x\mapsto\xh\in L^2(\Mcal,\tau)$ for the usual mapping with dense range.
For convenience, we will write the inner product on $L^2(\Mcal,\tau)$ to be linear in the second variable and conjugate linear in the first
variable.
Thus, we have, for $x_1,x_2\in\Mcal$,
\[
\langle \xh_1,\xh_2\rangle=\tau(x_1^*x_2).
\]
Then we have
$L^2(\Mcal,\tau)=L^2(\Mcal,F)\otimes_\Dc L^2(\Dc,\tau_\Dc)$,
and this is isomorphic to
\[
L^2(\Dc,\tau_\Dc)\oplus\bigoplus_{\substack{k\ge1 \\ i_1,\ldots,i_k\ge1 \\ i_j\ne i_{j+1}}}
\HEu_{i_1}\oup\otimes_\Dc\cdots\otimes_\Dc\HEu_{i_k}\oup\otimes_\Dc L^2(\Dc,\tau_\Dc),
\]
where $\HEu_i\oup$ is the Hilbert $\Dc,\Dc$-bimodule $L^2(\Bc_i,E_i)\ominus\Dc$.
We will denote by $\lambda$ the left action of $\Mcal$ on $L^2(\Mcal,\tau)$ and by $\rho$ the anti-multiplicative
right action,
$\rho(x)=J\lambda(x^*)J$, where $J$ is the standard conjugate linear isometry of $L^2(\Mcal,\tau)$
defined by $\xh\mapsto(x^*)\hat{\;}$.

\begin{lemma}\label{lem:ip}
Let $N\in\Nats$, let
\[
\eta_1,\,\eta_2\in
L^2(\Dc,\tau_\Dc)\oplus\bigoplus_{\substack{k\ge1 \\ 1\le i_1,\ldots,i_k\le N \\ i_j\ne i_{j+1}}}
\HEu_{i_1}\oup\otimes_\Dc\cdots\otimes_\Dc\HEu_{i_k}\oup\otimes_\Dc L^2(\Dc,\tau_\Dc)
\]
and let $b_1,b_2\in\Bc_{N+1}$.
Let $c_1,c_2,d_1,d_2\in\Dc$ be such that 
\[
c_1^*c_2=E_{N+1}(b_1^*b_2),\qquad d_2d_1^*=E_{N+1}(b_2b_1^*).
\]
Then
\begin{align*}
\langle\lambda(b_1)\eta_1,\lambda(b_2)\eta_2\rangle
&=\langle\lambda(c_1)\eta_1,\lambda(c_2)\eta_2\rangle, \\
\langle\rho(b_1)\eta_1,\rho(b_2)\eta_2\rangle
&=\langle\rho(d_1)\eta_1,\rho(d_2)\eta_2\rangle.
\end{align*}
\end{lemma}
\begin{proof}
We may without loss of generality assume $\eta_j=\xh_j$ for some $x_j\in W^*(\bigcup_{j=1}^N\Bc_j)$.
Then
\[
\langle\lambda(b_1)\eta_1,\lambda(b_2)\eta_2\rangle=\tau(x_1^*b_1^*b_2x_2)
=\tau_\Dc(F(x_1^*b_1^*b_2x_2)).
\]
By freeness, we have
\[
F(x_1^*b_1^*b_2x_2)=F(x_1^*F(b_1^*b_2)x_2)=F(x_1^*c_1^*c_2x_2),
\]
from which we get
\[
\langle\lambda(b_1)\eta_1,\lambda(b_2)\eta_2\rangle=\tau(x_1^*c_1^*c_2x_2)
=\langle\lambda(c_1)\eta_1,\lambda(c_2)\eta_2\rangle.
\]
Similarly, we have
\begin{multline*}
\langle\rho(b_1)\eta_1,\rho(b_2)\eta_2\rangle=\tau(b_1^*x_1^*x_2b_2)=\tau(x_2b_2b_1^*x_1^*) \\
=\tau(x_2d_2d_1^*x_1^*)
=\langle\rho(d_1)\eta_1,\rho(d_2)\eta_2\rangle.
\end{multline*}
\end{proof}

\begin{thm}\label{thm:ZM}
The center of $\Mcal$ lies in $\Dc$.
In particular,
\begin{equation}\label{eq:ZM}
Z(\Mcal)=\Dc\cap Z(\Bc).
\end{equation}
\end{thm}
\begin{proof}
It suffices to show $Z(\Mcal)\subseteq\Dc$, for then~\eqref{eq:ZM} follows readily.

Let $x\in Z(\Mcal)$.
Let $\eta=\xh-F(x)\hat{\;}$.
Then
\[
\eta\in\bigoplus_{\substack{k\ge1 \\ i_1,\ldots,i_k\ge1 \\ i_j\ne i_{j+1}}}
\HEu_{i_1}\oup\otimes_\Dc\cdots\otimes_\Dc\HEu_{i_k}\oup\otimes_\Dc L^2(\Dc,\tau_\Dc).
\]
For $N\in\Nats$, let $\eta_N$ be the orthogonal projection of $\eta$ onto the subspace
\[
\bigoplus_{\substack{k\ge1 \\1\le i_1,\ldots,i_k\le N \\ i_j\ne i_{j+1}}}
\HEu_{i_1}\oup\otimes_\Dc\cdots\otimes_\Dc\HEu_{i_k}\oup\otimes_\Dc L^2(\Dc,\tau_\Dc).
\]
Then $\eta_N$ converges in $L^2(\Mcal,\tau)$ to $\eta$ as $N\to\infty$.
Suppose $b\in\Bc\cap\ker E$.
Fix $N\in\Nats$ and
let $b_N$ denote the copy of $b$ in the copy $\Bc_N\subseteq\Mcal$ of $\Bc$.
Then $\lambda(b_N)\eta_{N-1}$ and $\rho(b_N)\eta_{N-1}$ are orthogonal to each other, because they lie in the respective
subspaces
\begin{gather}
\bigoplus_{\substack{k\ge1 \\1\le i_1,\ldots,i_k\le N-1 \\ i_j\ne i_{j+1}}}
\HEu_N\oup\otimes_\Dc\HEu_{i_1}\oup\otimes_\Dc\cdots\otimes_\Dc\HEu_{i_k}\oup\otimes_\Dc L^2(\Dc,\tau_\Dc),  
\label{eq:subspLeft} \\
\bigoplus_{\substack{k\ge1 \\1\le i_1,\ldots,i_k\le N-1 \\ i_j\ne i_{j+1}}}
\HEu_{i_1}\oup\otimes_\Dc\cdots\otimes_\Dc\HEu_{i_k}\oup\otimes_\Dc\HEu_N\oup\otimes_\Dc L^2(\Dc,\tau_\Dc).
\label{eq:subspRight} 
\end{gather}
On the other hand, $\lambda(b_N)F(x)\hat{\:}$ and $\rho(b_N)F(x)\hat{\:}$ lie in the subspace
$\HEu_N\oup\otimes_\Dc L^2(\Dc,\tau_\Dc)$, which is orthogonal to both
of the subspaces~\eqref{eq:subspLeft} and~\eqref{eq:subspRight}.
Thus, we have
\begin{multline*}
0=(b_N x-xb_N)\hat{\;}=\big(\lambda(b_N)-\rho(b_N)\big)\xh \\
=\big(\lambda(b_N)-\rho(b_N)\big)\big(\eta_{N-1}+F(x)\hat{\;}+(\eta-\eta_{N-1})\big)
\end{multline*}
and from the orthogonality relations noted above, we get
\begin{align}
\|\lambda(b_N)\eta_{N-1}&\|_2^2+\|\rho(b_N)\eta_{N-1}\|_2^2  \notag \\
&\le\big\|\lambda(b_N)\eta_{N-1}-\rho(b_N)\eta_{N-1}+(\lambda(b_N)-\rho(b_N))F(x)\hat{\;}\big\|_2^2 \notag \\
&=\|(\lambda(b_N)-\rho(b_N))(\eta-\eta_{N-1})\|_2^2\label{eq:2nms} \\
&\le4\|b\|^2\;\|\eta-\eta_{N-1}\|_2^2. \notag
\end{align}
Consider the elements $d_1=E(b^*b)^{1/2}$ and $d_2=E(bb^*)^{1/2}$ of $\Dc$.
By Lemma \ref{lem:ip}, we have
\[
\|\lambda(b_N)\eta_{N-1}\|_2=\|\lambda(d_1)\eta_{N-1}\|_2,\qquad
\|\rho(b_N)\eta_{N-1}\|_2=\|\rho(d_2)\eta_{N-1}\|_2
\]
and from~\eqref{eq:2nms}, we get
\[
\|\lambda(d_1)\eta_{N-1}\|_2^2+\|\rho(d_2)\eta_{N-1}\|_2^2\le4\|b\|^2\|\eta-\eta_{N-1}\|_2^2.
\]
Letting $N\to\infty$, we get
\begin{equation}\label{eq:lrd0}
\lambda(d_1)\eta=0=\rho(d_2)\eta.
\end{equation}
Let $q=q(E)\in\Dc\cap Z(\Bc)$ be the projection associated to the conditional expectation $E:\Bc\to\Dc$
as described in Lemma~\ref{lem:suppProj}.
From~\eqref{eq:lrd0} and letting $b$ run through all of $\ker E$, we get $\lambda(q)\eta=\rho(q)\eta=0$.
This yields
$q(x-F(x))=0$,
so $x-F(x)\in(1-q)\Bc=(1-q)\Dc$.
But $x-F(x)\perp\Dc$, so we must have $x-F(x)=0$ and $x\in\Dc$.
\end{proof}

The aim of the remainder of this section
(realized in Corollary~\ref{cor:traces}, below) is to characterize the normal traces on a von Neumann subalgebra 
whose compositions with a given conditional expectation are traces on the larger von Neumann algebra.
The result is quite natural and is perhaps known.
It may also be possible to prove it directly using state decompositions or averaging techniques, rather than free products.
However, as we get it from the results above with very little extra effort, it seems worth doing it here.
Furthermore, it is clearly related to the proof of our main result, Theorem~\ref{thm:choq}, and indeed
to the improved characterization of extremality of elements of $\TQSS(A)$,
though we don't actually use it in the proof.

Let $\Dc\subseteq\Bc$ be a unital inclusion of finite von Neumann algebras with 
a faithful conditional expectation $E:\Bc\to\Dc$.
Suppose there is a normal faithful tracial state $\rho$ on $\Dc$ such that $\rho\circ E$ is a trace on $\Bc$.
Let
\begin{equation}\label{eq:C}
\Cc=Z(\Bc)\cap\Dc.
\end{equation}
Let $(\Mcal,F)$ be the free product of infinitely many copies of $(\Bc,E)$ with amalgamation over $\Dc$,
as in~\eqref{eq:MF1}.
Due to the existence of $\rho$, by Proposition~\ref{prop:amalgtr}, $\Mcal$ is a finite von Neumann algebra.
Let $\eta$ be the center-valued trace on $\Mcal$ and let $\eta\restrict_\Dc$ denote its restriction to $\Dc$.
By Theorem~\ref{thm:ZM}, the center of $\Mcal$ is $\Cc$ as in~\eqref{eq:C}.

Let $\alpha$ be a permutation of $\Nats$ that has no proper, nonempty, invariant subsets;
thus, $\alpha$ results from the shift on $\Ints$ after fixing a bijection from $\Nats$ to $\Ints$.
Let $\alphah$ be the automorphism of $\Mcal$ that permutes the copies of $\Bc$
in the free product construction~\eqref{eq:MF1} according to $\alpha$.
\begin{lemma}\label{lem:etaalpha}
We have $\eta=\eta\circ\alphah$
\end{lemma}
\begin{proof}
Dixmier averaging says that for any $x\in\Mcal$, $\eta(x)$ is
the unique element in the intersection of  $\Cc$ and the norm closed convex hull of
the unitary conjugates of $x$.
(See, for example, Section 8.3 of~\cite{KR83}).
In symbols, this is
\[
\{\eta(x)\}=\Cc\cap\overline{\conv\{uxu^*\mid u\in \Uc(\Mcal)\}}.
\]
Since $\Cc\subseteq\Dc$, $\alphah$ leaves every element of $\Cc$ fixed.
Thus,
\begin{align*}
\{\eta(x)\}&=\alphah(\{\eta(x)\})=\alphah(\Cc)\cap\alphah(\overline{\conv\{uxu^*\mid u\in \Uc(\Mcal)\}}) \\
&=\Cc\cap\overline{\conv\{u\,\alphah(x)u^*\mid u\in \Uc(\Mcal)\}})
=\{\eta(\alphah(x))\}.
\end{align*}
\end{proof}

\begin{lemma}\label{lem:etaE}
We have $\eta=\eta\circ F$.
\end{lemma}
\begin{proof}
It is well known and not difficult to check that for all $x\in\Mcal$, the ergodic averages
\[
\frac1n\sum_{k=0}^{n-1}\alphah^k(x)
\]
converge in $\|\cdot\|_2$-norm as $n\to\infty$ and, thus, also in strong operator topology, to $F(x)$.
Because the center valued trace is normal, using Lemma~\ref{lem:etaalpha}, we get
\[
\eta(F(x))=\lim_{n\to\infty}\frac1n\sum_{k=0}^{n-1}\eta(\alphah^k(x))=\eta(x).
\]
\end{proof}

For a von Neumann algebra $\Nc$, we let $\NTS(\Nc)$ denote the set of normal tracial states on $\Nc$.
\begin{cor}\label{cor:traces}
The map
\begin{equation}\label{eq:taumap}
\tau\mapsto \tau\circ\eta\restrict_\Dc
\end{equation}
is a bijection
from $\NTS(Z(\Bc)\cap\Dc)$ onto
\begin{equation}\label{eq:rhos}
\{\rho\in\NTS(\Dc)\mid\rho\circ E\text{ a trace on }\Bc\}.
\end{equation}
\end{cor}
\begin{proof}
It is clear that the map~\eqref{eq:taumap} is injective.

We view $\Bc$ as embedded in $\Mcal$ by identification of $\Bc$ with any of the copies arising in the free
product construction~\eqref{eq:MF1}.
Since, by Lemma~\ref{lem:etaE}, $\eta=\eta\circ E=\eta\restrict_\Dc\circ E$,
if $\tau\in\NTS(\Cc)$ and $\rho=\tau\circ\eta\restrict_\Dc$,
then $\rho\circ E=\tau\circ(\eta\restrict_\Bc)=(\tau\circ\eta)\restrict_\Bc$
is a trace on $\Bc$.
Thus, the map~\eqref{eq:taumap} goes into the set~\eqref{eq:rhos}.

To see that it is onto, 
suppose $\rho$ belongs to the set~\eqref{eq:rhos}.
Since $\Mcal$ is a finite von Neumann algebra, by a standard theory (see, for example, Theorem 8.3.10 of~\cite{KR83}),
the map $\tau\mapsto\tau\circ\eta$ is a bijection from $\NTS(\Cc)$ onto $\NTS(\Mcal)$.
By Proposition~\ref{prop:amalgtr} $\rho\circ F$ is a normal tracial state on $\Mcal$, so equals $\tau\circ\eta$ for some
$\tau\in\NTS(\Cc)$.
Thus, $\rho=\rho\circ F\restrict_\Dc=\tau\circ\eta\restrict_\Dc$, as required.
\end{proof}

\section{The conditional expectation onto the tail algebra in an amalgamated free product}
\label{sec:tail}

For a symmetric state $\psi$ on the universal free product $C^*$-algebra $\Afr=*_1^\infty A$,
we let $\Mcal_\psi$ denote the von Neumann algebra generated by the image of $\Afr$ under the GNS representation $\pi_\psi$
of $\Afr$ on $L^2(\Afr,\psi)$
arising from $\psi$ and let $\psih$ denote the normal extension of $\psi$ to $\Mcal_\psi$, which is the vector state for the vector
of $L^2(\Afr,\psi)$ corresponding to the identity element of $\Afr$.
The tail algebra $\Tc_\psi$ is the von Neumann subalgebra
\[
\Tc_\psi=\bigcap_{n\ge1}W^*(\bigcup_{k\ge n}\lambda_k(A))\subseteq\Mcal_\psi
\]
where $\lambda_k:A\to\Afr$ is the embedding onto the $k$-th copy of $A$ in the universal free product $C^*$-algebra.
Note that the action of the permutation group $S_\infty$ on $\Afr$ by permuting the embedded copies of $A$ results in a
$\psi$-preserving action of $S_\infty$ on $\Mcal_\psi$;
we let $\Fc_\psi$ denote the fixed point subalgebra of this action and we always have $\Tc_\psi\subseteq\Fc_\psi$
(see Lemma 5.1.3 of~\cite{DKW}).
By Proposition 5.2.4 of~\cite{DKW}, if the restriction of $\psih$ to $\Tc_\psi$ has faithful GNS representation
(in particular, if $\psih$ is faithful on $\Mcal_\psi$), then there is a normal, $\psih$-preserving conditional expectation
$E_\psi:\Mcal_\psi\to\Tc_\psi$;
furthermore, if also the restriction of $\psih$ to $\Fc_\psi$ has faithful GNS representation
(in particular, if $\psih$ is faithful on $\Mcal_\psi$), then $\Tc_\psi=\Fc_\psi$. 

\begin{prop}\label{prop:tail}
Let $\Dc\subseteq\Bct$ be a unital von Neumann subalgebra with $\Et:\Bct\to\Dc$ a normal, faithful, conditional 
expectation.
Let
\[
(\Mcalt,\Ft)\cong(*_\Dc)_1^\infty(\Bct,\Et)
\]
be the amalgamated free product of von Neumann algebras.
Suppose $\rho$ is a normal faithful state on $\Dc$.
Suppose $A$ is a unital C$^*$-algebra and $\sigma:A\to\Bct$ is a unital $*$-homomorphism.
Let $\psi=\rho\circ\Ft\circ(*_1^\infty\sigma):\Afr=*_1^\infty A\to\Cpx$.
By Proposition~3.1 of~\cite{DKW}, $\psi\in\QSS(A)$.
Then $\Mcal_\psi$
is canonically identified with a von Neumann subalgebra of $\Mcalt$ with the tail algebra $\Tc_\psi$ identified with a subalgebra of $\Dc$.
Moreover, the normal state $\psih$ on $\Mcal_\psi$ is identified with the restriction of the state $\rho\circ\Ft$ to $\Mcal_\psi$,
which is faithful,
and the normal conditional expectation $E_\psi:\Mcal_\psi\to\Tc_\psi$ is identified with the restriction to $\Mcal_\psi$ of $\Ft$.
\end{prop}
\begin{proof}
Note that under the hypotheses, $\rho\circ\Ft$ is a faithful state on $\Mcalt$ (by Ueda's result~\cite{U99},
as discussed in Section~\ref{sec:amalgtr} above).
Thus, the GNS Hilbert space $L^2(\Afr,\psi)$ is a subspace of $L^2(\Mcalt,\rho\circ\Ft)$
and $\Mcal_\psi$ is realized as the strong operator topology closure in $\Mcalt$ of $(*_1^\infty\sigma)(\Afr)$
with $\psih$ the restriction to $\Mcal_\psi$ of $\rho\circ\Ft$.
Now, by examining the free product structure of the Hilbert space $L^2(\Mcalt,\rho\circ\Ft)$,
we see that the fixed point subalgebra $\Fc_\psi$ must lie in $\Dc$, and since $\psih$ is faithful on $\Mcal_\psi$,
we have $\Tc_\psi=\Fc_\psi\subseteq\Dc$.

We must only show that the conditional expectation $E_\psi:\Mcal_\psi\to\Tc_\psi$ equals the restriction to $\Mcal_\psi$
of $\Ft$.
Since both of these conditional expectations are normal, it will suffice to show their agreement on elements of $\pi_\psi(\Afr)$.
For this, we appeal to the construction of the conditional expectation $G_\psi$ found in Theorem 5.1.10 of~\cite{DKW};
since $\psih$ is faithful on $\Mcal_\psi$, this conditional expectation $G_\psi$ coincides with the restriction to $\pi_\psi(\Afr)$
of $E_\psi$.
The $*$-endomorphism $\alpha$ appearing in the aforementioned construction of $G_\psi$ must, by Lemma 5.1.9 of~\cite{DKW},
agree with the normal ``shift'' $*$-endomorphism $\alphat$ of $\Mcalt$, that sends the $i$-th copy of $\Bct$ in $\Mcalt$ to the $(i+1)$-st
copy (and which is easily seen to exist, by the construction outlined in Section~\ref{sec:amalgtr}).
Thus, (see Theorem 5.1.10 of~\cite{DKW}), 
\[
E_\psi(x)={\rm WOT-}\lim_{n\to\infty}\alphat^n(x)
\]
for all $x\in\pi_\psi(\Afr)$, and by the structure of the free product Hilbert space $L^2(\Mcalt,\rho\circ\Ft)$,
we conclude $E_\psi(x)=\Ft(x)$.
\end{proof}

\begin{remark}\label{rmk:tailgen}
In the situation of the previous proposition, by the methods of Section~7
of~\cite{DKW}
(see in particular Theorem 7.3 of~\cite{DKW})
the tail algebra of $\psi$ is equal to the smallest von Neumann subalgebra $\Dc_\infty$ of $\Dc$
that contains
\begin{equation}\label{eq:Fa}
\Ft(\sigma(a_1)d_1\sigma(a_2)\cdots d_{n-1}\sigma(a_n))
\end{equation}
for every $a_1,\ldots a_n\in A$
and every $d_1,\ldots,d_{n-1}\in\Dc_\infty$.
Thus, letting $\Dc_0=\Cpx1$ and for $p\ge1$ letting $\Dc_p$ be the von Neumann algebra generated by
all expressions of the form~\eqref{eq:Fa} for $a_j\in A$ and $d_1,\ldots,d_{n-1}\in\Dc_{p-1}$,
we have that $\Dc_\infty$ equals the von Neumann algebra generated by $\bigcup_{p=0}^\infty\Dc_p$.
\end{remark}

\section{The simplex of tracial quantum symmetric states}
\label{sec:Choq}

Let $A$ be a unital C$^*$-algebra and let $\TQSS(A)$
be the compact, convex set of tracial, quantum symmetric states
on $\Afr=*_1^\infty A$.
We assume that $A$ has a tracial state,
so that $\TQSS(A)$ is nonempty, and we assume that $A\ne\Cpx$.

By
Theorem~7.6
of~\cite{DKW},
$\TQSS(A)$ is in bijection with the set of (equivalence classes of) quintuples
$(\Bc,\Dc,E,\sigma,\rho)$ where $E:\Bc\to\Dc\subseteq\Bc$ is a faithful conditional expectation of von Neumann algebras,
$\sigma:A\to\Bc$ is an injective, unital $*$-homomorphism and $\rho$ is a normal faithul, tracial state on $\Dc$
such that $\rho\circ E$ is a trace on $\Bc$, and certain minimality conditions are satisfied.
These minimality conditions are that $\Bc$ is generated by $\Dc\cup\sigma(A)$ and $\Dc$ is the smallest unital von Neumann
subalgebra of $\Bc$ that satisfies $E(d_0\sigma(a_1)d_1\cdots\sigma(a_n)d_n)\in\Dc$ whenever $n\in\Nats$, $d_0,\ldots,d_n\in\Dc$
and $a_1,\ldots,a_n\in A$.
Given a quintuple $(\Bc,\Dc,E,\sigma,\rho)$, one constructs the amalgamated free product von Neumann algebra
\begin{equation}\label{eq:MF}
(\Mcal,F)=(*_\Dc)_1^\infty(\Bc,E)
\end{equation}
of infinitely many copies of $(\Bc,E)$ and one takes the free product $*$-homo\-morph\-ism $*_1^\infty\sigma:\Afr\to\Mcal$ arising from
the universal property, sending the $i$-th copy of $A$ into the $i$-th copy of $\Bc$.
The tracial state $\psi=\rho\circ F\circ(*_1^\infty\sigma)$ on $\Afr$
is the tracial quantum symmetric state of $A$ that corresponds 
to $(\Bc,\Dc,E,\sigma,\rho)$ under the bijection refered to above.
Then $\Dc=\Tc_\psi$ is the tail algebra and $\Mcal=\Mcal_\psi$ is the von Neumann algebra generated by the GNS representation of $\psi$.

The extreme points of $\TQSS(A)$ were characterized in Theorem~8.2
of~\cite{DKW} as corresponding to the set of quintuples
$(\Bc,\Dc,E,\sigma,\rho)$ so that $\rho$ is extreme among the set $R(E)$ of tracial states of $\Dc$ so that $\rho\circ E$ is a trace
on $\Bc$.
In fact, we arrive at a better characterization of the extreme tracial quantum symmetric states below.

Note that $\TQSS(A)$ is a closed convex subset of the tracial state space, $TS(\Afr)$, of $\Afr$.
The tracial state space of any C$^*$-algebra is known to be a Choquet simplex (see, for example Theorem 3.1.18
of~\cite{Sa71})
and the extreme points of it are the tracial states that are factor states.

\begin{thm}\label{thm:choq}
$\TQSS(A)$ is a Choquet simplex and is a face of $TS(\Afr)$.
Moreover, for $\psi\in\TQSS(A)$ with 
corresponding quintuple $(\Bc,\Dc,E,\sigma,\rho)$, 
the following are equivalent:
\begin{enumerate}[(i)]
\item $\psi$ is an extreme point of $\TQSS(A)$
\item $\psi$ is an extreme point of $TS(\Afr)$
\item $\Dc\cap Z(\Bc)=\Cpx1$.
\end{enumerate}
\end{thm}
\begin{proof}
The implication (i)$\implies$(ii), when proved, will imply that $\TQSS(A)$ is
a face of $TS(\Afr)$ and, thus, a Choquet simplex.

The implication (ii)$\implies$(i) is clearly true.

Let $(\Mcal,F)$ be as in~\eqref{eq:MF}.
By
Theorem~\ref{thm:ZM}, condition~(iii) is equivalent to factoriality of $\Mcal$,
and this is equivalent to condition~(ii).
Thus, conditions~(ii) and~(iii) are equivalent.

To finish the proof, it will suffice to show (i)$\implies$(iii).
If~(iii) fails to hold, then $\Dc\cap Z(\Bc)$ has a projection $p$ equal to neither $0$ nor $1$.
Let $t=\rho(p)$.
Since $\rho$ is faithful, we have
$0<t<1$ and we can write $\rho=t\rho_0+(1-t)\rho_1$, where
\[
\rho_0(x)=t^{-1}\rho(px),\qquad\rho_1(x)=(1-t)^{-1}\rho((1-p)x).
\]
Since $p$ lies in $\Dc\cap Z(\Bc)$, we see that $\rho_0$ and $\rho_1$ are distinct normal tracial states on $\Dc$
and that $\rho_i\circ E$ is a trace on $\Bc$ ($i=0,1$).
Thus, $\rho$ is not an extreme point of $R(E)$, and $\psi$ is not extreme in $\TQSS(A)$.
\end{proof}

In Theorem~\ref{thm:Poulsen},
we will use multiplicative free Brownian motion (see~\cite{B97}) to show that every quantum symetric state is a limit
of extreme quantum symmetric states.
This will show that $\TQSS(A)$ is the Poulsen simplex, when $A$ is separable and not a copy of $\Cpx$.

Multiplicative free Brownian motion is the solution $(U_t)_{t\ge0}$ of the linear stochastic differential equation 
\[
U_t=1-\frac{1}{2}\int_0^tU_sds+\int_0^tidS_sU_s=e^{-t/2}+\int_0^tidS_se^{-(t-s)/2}U_s,
\]
where $(S_t)_{t\ge0}$ is an additive free Brownian motion.
Then each $U_t$ is unitary (see~\cite{B97a}) and belongs to the von Neumann algebra
$W^*(S_t,t>0)$, which is a copy of $L(\Fbb_\infty)$.
We will need the following lemma.

\begin{lemma}\label{lem:MBM}
Let $\Mcal$ be a von Neumann algebra with normal, faithful, tracial state $\tau$ and suppose $\Nc\subseteq\Mcal$ is a unital
von Neumann subalgebra and $(U_t)_{t\ge0}$ is a multiplicative free Brownian motion that is free from $\Nc$ with respect to $\tau$.
Then for every unital C$^*$-subalgebra $A\subseteq\Nc$ with $\dim(A)>1$ and for every $t>0$,
we have
\[
(U_t^*A U_t)'\cap\Nc=\Cpx1.
\]
\end{lemma}
\begin{proof}
If $(U_t^*A U_t)'\cap\Nc$ is nontrivial, then it contains a projection $p\notin\{0,1\}$.
Without loss of generality, we may assume $A$ is a von Neumann subalgebra of $\Nc$ and, thus,
contains a projection $q\notin\{0,1\}$.

From Proposition 9.4 and Remark 8.10 of~\cite{Vo99}, the liberation Fisher information satisfies
\[
\varphi^*(U_t^*A U_t:\Nc)\leq F(U_t)<\infty,
\]
for any $t>0$, where $F$ is the Fisher information for unitaries.
Thus, from Remark 9.2(e) of~\cite{Vo99}, we have
\[
\varphi^*(W^*(U_t^*qU_t) :W^*(p))\leq \varphi^*(U_t^*A U_t:\Nc)<\infty.
\]
As a consequence, the assumptions of Lemma 12.5 of~\cite{Vo99} are satisfied and, therefore,
$U_t^*qU_t$ and $p$ are in general position,
i.e.,
\begin{equation}\label{eq:genpos1}
U_t^*qU_t\wedge p=0\quad\text{or}\quad U_t^*(1-q)U_t\wedge (1-p)=0,
\end{equation}
and
\begin{equation}\label{eq:genpos2}
U_t^*(1-q)U_t\wedge p=0\quad\text{or}\quad U_t^*qU_t\wedge (1-p)=0.
\end{equation}
But this is not compatible with the assumption that $U_t^*qU_t$ and $p$ commute.
For example, if
\[
U_t^*qU_t\wedge p=U_t^*(1-q)U_t\wedge p=0,
\]
then
\[
0=U_t^*qU_tp+U_t^*(1-q)U_tp=p,
\]
contrary to hypothesis,
and similarly if other cases of~\eqref{eq:genpos1} and~\eqref{eq:genpos2} hold.
\end{proof}

\begin{thm}\label{thm:Poulsen}
For every unital C$^*$-algebra $A$ with $\dim(A)>1$,  the extreme points of $\TQSS(A)$ are dense in $\TQSS(A)$.
Hence, if $A$ is also separable, then $\TQSS(A)$
is the Poulsen simplex.
\end{thm}
\begin{proof}
If $A$ is separable, then the free product algebra $\Afr$ is also separable and, thus, $\TQSS(A)$ is second countable.
By Urysohn's metrization theorem, it is metrizable.
Once the density of extreme points is shown, it will follow that $\TQSS(A)$ is the Poulsen simplex (see~\cite{LOS78}).

We now show density of extreme points.
Let $\psi\in\TQSS(A)$ and let $(\Bc,\Dc,E,\rho,\sigma)$ be its associated quintuple.
We use the notation from the description at the beginning of this section.
In particular, $\psi=\rho\circ F\circ(*_1^\infty\sigma)$, and we let $\psih=\rho\circ F$ denote the normal
extension of $\psi$ to $\Mcal$.
Let
\[
(\Mcalt,\tau)=(\Mcal,\psih)*(L(\Fbb_\infty),\tau_{\Fbb_\infty})
\]
be the free product of $\Mcal$ with a copy of $L(\Fbb_\infty)$.
Then, since $(L(\Fbb_\infty),\tau_{\Fbb_\infty})\cong*_1^\infty (L(\Fbb_\infty),\tau_{\Fbb_\infty})\cong*_1^\infty (W^*(S_t,t>0),\tau)$, for the von Neumann algebra of a free Brownian motion algebra $W^*(S_t,t>0)\cong L(\Fbb_\infty)$,
 letting
\[
(\Bct,\eta)=(\Bc,\rho\circ E)*(W^*(S_t,t>0),\tau)
\]
and letting $\Et:\Bct\to\Dc$ be the composition of the
$\eta$-preserving conditional expectation $\Bct\to\Bc$
arising from the free product construction with the conditional expectation
$E:\Bc\to\Dc$,
we have that $\Mcalt$ is isomorphic to the von Neumann algebra free product with amalgamation,
\begin{equation}\label{eq:MtFt}
(\Mcalt,\Ft)\cong(*_\Dc)_1^\infty(\Bct,\Et)
\end{equation}
and the trace $\tau$ arises as $\rho\circ\Ft$.

Letting $(U_t)_{t\ge0}$ be a multiplicative free Brownian motion in $W^*(S_t,t>0)$,
from the free $L^\infty$ version of the Burkholder-Gundy inequalities (Theorem 3.2.1 of~\cite{BS}),
we have the upper bound 
\begin{multline}\label{Boundat0}
||U_t-1||\leq (1-e^{-t/2})+2\sqrt{2}\left(\int_0^t||U_s||^2e^{-(t-s)}ds\right)^{1/2} \\
=(1-e^{-t/2})+2\sqrt{2(1-e^{-t})},
\end{multline}
which tends to zero as $t\to0^+$.

Let $\sigma_t:A\to\Bct$ be the $*$-homomorphism $U_t\sigma(\cdot)U_t^*$.
Then $*_1^\infty\sigma_t$ is a $*$-homomorphism from $\Afr$ into $\Mcalt$.
By freeness with amalgamation (see Proposition~3.1 of~\cite{DKW}), 
the state $\psi_t:=\rho\circ\Ft\circ(*_1^\infty\sigma_t)=\tau\circ(*_1^\infty\sigma_t)$ is a quantum symmetric state.

We will show that for every $t>0$, $\psi_t$ is an extreme point of $\TQSS(A)$.
By Proposition~\ref{prop:tail}, the tail algebra $\Tc_{\psi_t}$ of $\psi_t$ is a von Neumann subalgebra of $\Dc$,
and the conditional expectation $E_{\psi_t}$ onto the tail algebra is the restriction of $\Ft$.
In particular, see Remark~\ref{rmk:tailgen} for description of generators for $\Dc$.
Let $(\Bc_t,\Dc_t,E_t,\rho_t,\sigma_t)$ denote the quintuple corresponding to the quantum symmetric state $\psi_t$.
Then $\Dc_t=\Tc_{\psi_t}\subseteq\Dc$ and $\Bc_t\supseteq\sigma_t(A)$.
By Theorem~\ref{thm:choq}, showing that $\psi_t$ is an extreme point of $\TQSS(A)$
is equivalent to showing that $\Dc_t\cap Z(\Bc_t)$
is trivial.
But $\Dc_t\cap Z(\Bc_t)$ is contained in $\Dc_t\cap (U_t^*\sigma(A) U_t)'$.
By Lemma~\ref{lem:MBM}, the latter set is trivial,
and we have proved that $\psi_t$ is an extreme tracial quantum symmetric state.

From the bound \eqref{Boundat0}, we deduce that  for every $x\in\Afr$,
$\lim_{t\to0^+}\|\psi_t(x)-\psi(x)\|=0$, working first
with the case of $x$ in the algebraic free product,
and passing to the general case by norm approximation.
\end{proof}

\begin{remark}
In contrast, the simplices $\ZQSS(A)$ and $\ZTQSS(A)$ of central quantum symmetric
states and central tracial quantum symmetric states, respectively, (see~\cite{DKW})
are Bauer simplices, meaning that their respective sets of extreme points are closed.
This follows from the proof of Theorem~9.2 of~\cite{DKW} and in particular the fact that the map $\phi\mapsto*_1^\infty\phi$
in equation~(35)
of~\cite{DKW} is a homeomorphism from $S(A)$ onto the extreme boundary of $\ZQSS(A)$ and, by restricting to the tracial state
space, yields a homeomorphism from $TS(A)$ onto the extreme boundary of $\ZTQSS(A)$.
\end{remark}


\begin{bibdiv}
\begin{biblist}

\bib{B97a}{article}{
   author={Biane, Philippe},
   title={Free Brownian motion, free stochastic calculus and random matrices},
   conference={
      title={Free probability theory},
      address={Waterloo, ON},
      date={1995},
   },
   book={
      series={Fields Inst. Commun.},
      volume={12},
      publisher={Amer. Math. Soc.},
      place={Providence, RI},
   },
   date={1997},
   pages={1--19},
}

\bib{B97}{article}{
   author={Biane, Philippe},
   title={Segal-Bargmann transform, functional calculus on matrix spaces and the theory of semi-circular and circular systems},
   journal={J. Funct. Anal.},
   volume={144},
   date={1997},
   pages={232--286},
}
		
\bib{BS}{article}{
   author={Biane, Philippe},
   author={Speicher, Roland},
   title={Stochastic calculus with respect to free Brownian motion and analysis on Wigner space},
   journal={Probab. Theory Related Fields},
   volume={112},
   date={1998},
   pages={373--409},
}

\bib{DKW}{article}{
  author={Dykema, Ken},
  author={K\"ostler, Claus},
  author={Williams, John},
  title={Quantum symmetric states on free product C$^*$-algebras},
  eprint={http://arxiv.org/abs/1305.7293},
}

\bib{KR83}{book}{
  author={Kadison, Richard V.},
  author={Ringrose, John R.},
  title={Fundamentals of the Theory of Operator Algebras, vols. I and II},
  publisher={Academic Press},
  year={1983}
}

\bib{KS09}{article}{
  author={K\"ostler, Claus},
  author={Speicher, Roland},
  title={A noncommutative de Finetti theorem:
         invariance under quantum permutations is equivalent to freeness with amalgamation},
  journal={Comm. Math. Phys.},
  volume={291},
  year={2009},
  pages={473--490}
}

\bib{LOS78}{article}{
   author={Lindenstrauss, J.},
   author={Olsen, G.},
   author={Sternfeld, Y.},
   title={The Poulsen simplex},
   journal={Ann. Inst. Fourier (Grenoble)},
   volume={28},
   date={1978},
   pages={91--114},
}

\bib{Sa71}{book}{
  author={Sakai, Shoichiro},
  title={C$^*$-algebras and W$^*$-algebras},
  publisher={Springer--Verlag},
  year={1971}
}

\bib{U99}{article}{
   author={Ueda, Yoshimichi},
   title={Amalgamated free product over Cartan subalgebra},
   journal={Pacific J. Math.},
   volume={191},
   date={1999},
   pages={359--392},
}

\bib{Vo99}{article}{
   author={Voiculescu, Dan},
   title={The analogues of entropy and of Fisher's information measure in free probability theory. VI. Liberation and mutual free information},
   journal={Adv. Math.},
   volume={146},
   date={1999},
   pages={101--166},
}

\bib{W98}{article}{
  author={Wang, S.},
  title={Quantum symmetry groups of finite spaces},
  journal={Comm. Math. Phys.},
  volume={195},
  pages={195--211},
  year={1998}
}

\end{biblist}
\end{bibdiv}
\end{document}